\documentclass[12pt]{article}
\usepackage[]{amsmath,amssymb,amsfonts,latexsym,amsthm,enumerate, fullpage}
\newtheorem{thm}{Theorem}
\newtheorem*{thm*}{Theorem}
\newtheorem{lemma}[thm]{Lemma}
\newtheorem*{lemma*}{Lemma}

\newtheorem*{prop*}{Proposition}

\newtheorem{fact}[thm]{Fact}
\newtheorem{conj}[thm]{Conjecture}
\newtheorem{claim}[thm]{Claim}
\theoremstyle{remark}

\newtheorem*{rmk*}{Remark}
\newtheorem*{rmks*}{Remarks}
\newtheorem*{not*}{Notation}
\newtheorem*{claim*}{Claim}
\newtheorem*{fact*}{Fact}
\newtheorem*{conj*}{Conjecture}

\newtheorem{dfn}{Definition}
\newtheorem*{dfn*}{Definition}

\def\C{\mathbb{C}}

\def\E{\mathbb{E}}
\def\F{\mathbb{F}}

\newcommand\ignore[1]{}
\newcommand\ip[1]{\left<#1\right>}

\begin{document}

\title{Equivalence of polynomial conjectures in additive combinatorics}
\author{Shachar Lovett
\thanks{Research supported by the Israel Science Foundation (grant 1300/05).}\\
The Weizmann Institute of Science\\
Faculty of Mathematics and Computer Science\\
POB 26, Rehovot 76100, Israel.\\
Email: shachar.lovett@weizmann.ac.il
}
\maketitle{}

\abstract{
We study two conjectures in additive combinatorics. The
first is the polynomial Freiman-Ruzsa conjecture, which relates to
the structure of sets with small doubling. The second is the inverse
Gowers conjecture for $U^3$, which relates to functions which locally look
like quadratics. In both cases a weak form, with exponential decay
of parameters is known, and a strong form with only a polynomial loss of
parameters is conjectured.
Our main result is that the two conjectures are in fact equivalent.
}

\section{Introduction}
Additive combinatorics studies subsets of abelian groups, with the
main examples are subsets of the integers and of vector spaces
over finite fields. The main problems entail connecting various
properties related to the additive structure of the space, to
structural properties of the subsets. In a way, additive
combinatorics can be viewed as a robust analog of basic linear
algebra.

We study in this paper two conjectures relating to objects defined
over vector spaces $\F^n$. The first is the polynomial
Freiman-Ruzsa conjecture, which relates to subsets $S \subset
\F^n$ which are approximately vector spaces. The second is the
polynomial inverse Gowers conjecture for the $U^3$ norm, which
relates to functions $f:\F^n \to \F$ which are approximately
quadratic. Both conjectures aim to give structural properties for these
objects.

Our main result is that the two conjectures are equivalent. We
focus in the paper on the case of $\F=\F_2$; our results extend
easily to any constant finite field $\F_p$.

\subsection{Approximate vector spaces}
Let $S \subset \F_2^n$. The set $S$ is said to have {\em doubling}
$K$ if $|S+S| \le K|S|$, where $S+S=\{x+y:x,y \in S\}$. It is
clear that $S$ has doubling $1$ iff it is an affine space. Thus, a
set with small doubling can be viewed as an approximate vector
space. Can we infer some structure such sets must have? The
following theorem of Ruzsa~\cite{Ruzsa93analog} claims that any
such set is contained in a vector space which is not much larger.

\begin{thm}[Theorem 1 in ~\cite{Ruzsa93analog}]\label{thm:ruzsa_small_doubling}
Let $S \subset \F_2^n$ such that $|S+S| \le K|S|$. Then $|Span(S)| \le K^2 2^{K^4} |S|$.
\end{thm}

The work of Ruzsa is an analog of a similar result of
Freiman~\cite{Freiman} for subsets of the integers with small
doubling. Theorem~\ref{thm:ruzsa_small_doubling} was improved in a
series of works (Green and Ruzsa~\cite{GreenRu06},
Sanders~\cite{Sanders08} and Green and
Tao~\cite{GreenTao09:Ruzsa}) to an almost optimal bound.

\begin{thm}[Theorem 1.3 in ~\cite{GreenTao09:Ruzsa}]\label{thm:greentao_ruzsa}
Let $S \subset \F_2^n$ such that $|S+S| \le K|S|$. Then $|Span(S)| \le 2^{(2+o(1))K} |S|$.
\end{thm}

The bound is tight up to the $o(1)$ term as can be seen by the following example: let $S=\{v_1,\ldots,v_r\}$
where $v_1,\ldots,v_r \in \F_2^n$ are linearly independent. We have $|S+S| \approx \tfrac{r}{2}|S|$ and $|Span(S)| = 2^r$. We could also have $S=V+\{v_1,\ldots,v_r\}$ where $V$ is a vector space and $v_1,\ldots,v_r \in V^{\perp}$.

This example shows that the exponential loss of parameters in Theorem~\ref{thm:greentao_ruzsa} is inevitable. It would be beneficial, however, to have some structure theorem for sets with small doubling which have only a polynomial loss of parameters.
In general, theorems which involve only a polynomial loss of parameters are useful as they can be
applied iteratively several times, resulting again with only a polynomial loss of parameters. The following strengthening of Theorems~\ref{thm:ruzsa_small_doubling} and~\ref{thm:greentao_ruzsa}, known as the {\em Polynomial
 Freiman-Ruzsa conjecture} was suggested in several works.

\begin{conj}[Polynomial Freiman-Ruzsa conjecture]\label{conj:pfr}
Let $S \subset \F_2^n$ such that $|S+S| \le K|S|$. Then there is a  subset $S' \subset S$, $|S'| \ge K^{-O(1)} |S|$, such that $|Span(S')| \le K^{O(1)} |S|$.
\end{conj}

Conjecture~\ref{conj:pfr} was proved by Green and Tao for the
special case when $S$ is a downset~\cite{GreenTao09:Ruzsa}, as
well as in the general case with an exponential loss of parameters
which is better than that given by
Theorem~\ref{thm:greentao_ruzsa}~\cite{GreenTao09:note}.

The Polynomial Freiman-Ruzsa conjecture can be equivalently
restated is several forms. We give below two such forms which
relate to approximate homomorphisms. For proofs of the equivalence
as well as several other equivalent formulations
see~\cite{Green05:finite_field_models}.

The first formulation relates to testing if a function $f:\F_2^n
\to \F_2^m$ is close to a linear map. A natural way to do so is to
sample $x,y \in \F_2^n$ and verify that $f(x+y)=f(x)+f(y)$. The
following conjecture states that if this event occurs with
polynomial $\epsilon$ over the choice of $x,y$, then $f$ is
$poly(\epsilon)$ close to a linear map.

\begin{conj}[Approximate homomorphism testing]\label{conj:approx_hom_testing}
Let $f:\F_2^n \to \F_2^m$ be such that $\Pr_{x,y}[f(x+y)=f(x)+f(y)] \ge \epsilon$. Then there is a linear map $\ell:\F_2^n \to \F_2^m$ such that $\Pr_x[f(x)=\ell(x)] \ge \epsilon^{O(1)}$.
\end{conj}

The second formulations relates to structured approximate homomorphisms. For a function $f:\F_2^n \to \F_2^m$ define its {\em difference set} $\Delta f = \{f(x+y)-f(x)-f(y): x,y \in \F_2^n\}$. The following conjectures claim that if $\Delta f$ is small then $f$ can be expressed as the sum of a linear function
and an error function, where the error function obtains at most $poly(|\Delta f|)$ possible values.
\begin{conj}[Structured approximate homomorphism]\label{conj:structured_approx_hom}
Let $f:\F_2^n \to \F_2^m$ and assume that $|\Delta f| \le K$. Then there is a linear map $\ell:\F_2^n \to \F_2^m$ such that $f(x)=\ell(x)+e(x)$ where
$|\{e(x): x \in \F_2^n\}| \le K^{O(1)}$.
\end{conj}

Note that a-priory, it seems that the assumption of
Conjecture~\ref{conj:structured_approx_hom} is much stronger than
that of Conjecture~\ref{conj:approx_hom_testing}; nevertheless,
the conjectures are equivalent. We also note that analogs of
Conjectures~\ref{conj:approx_hom_testing}
and~\ref{conj:structured_approx_hom} with exponential loss of
parameters follow from Theorem~\ref{thm:greentao_ruzsa}.

There is another natural definition for an approximate vector space; $S \subset \F_2$ is an approximate
vector space if for many pairs $x,y \in S$ we have $x+y \in S$. The following theorem due to Balog,
Szemer\`{e}di and Gowers~\cite{BalogSz94,Gowers98} shows this property is polynomially related to the case of small doubling.

\begin{thm}[Balog-Szemer\`{e}di-Gowers]\label{thm:bsg}
Let $S \subset \F_2^n$.
If $\Pr_{x,y \in S}[x+y \in S] \ge \epsilon$ then there is a subset $S' \subset S$, $|S'| \ge \epsilon^{O(1)} |S|$ such that $S'$ has small doubling, $|S'+S'| \le \epsilon^{-O(1)} |S'|$.
\end{thm}

\subsection{Approximate polynomials}
Let $f:\F_2^n \to \F_2$ be a function.
Define the derivative of $f$ in direction $y \in \F_2^n$ as
$f_y(x) = f(x+y)+f(x)$
\footnote{Over odd fields define $f_y(x) = f(x+y)-f(x)$.}. If $f$ is a degree $d$ polynomial then $f_y$ is a polynomial of degree at most $d-1$. Define iterated derivatives as
$$
f_{y_1,\ldots,y_d}(x) = (f_{y_1,\ldots,y_{d-1}})_{y_d}(x) = \sum_{I \subseteq [d]}f(x+\sum_{i \in I} y_i)
$$
and observe that $f$ is a polynomial of degree at most $d-1$ iff $f_{y_1,\ldots,y_d}(x) \equiv 0$
for all $y_1,\ldots,y_d \in \F_2^n$.
On the other hand, if $f$ is a random boolean function then $f_{y_1,\ldots,y_d}(x)$ is distributed close
to uniform over $\F_2$. Thus, a plausible definition for an approximate polynomial is a function $f$ for which
$\Pr_{x,y_1,\ldots,y_d}[f_{y_1,\ldots,y_d}(x)=0] \ge 1/2+\epsilon$. This is captured by the {\em Gowers norm},
defined originally by Gowers~\cite{Gowers98} in his seminal work on a new proof for Szemer\'{e}di's theorem.
The Gowers norm is defined over complex functions $F:\F_2^n \to \C$ (think of $F(x) = (-1)^{f(x)}$). Define the derivative of $F$ in direction $y \in \F_2^n$ as $F_y(x) = F(x+y) \overline{F(x)}$, and iterated derivatives analogously.

\begin{dfn}[Gowers norm]
Let $F:\F_2^n \to \C$. The $d$-th Gowers norm of $F$ is defined as
$$
\|F\|_{U^d} = \left( \E_{x,y_1,\ldots,y_d \in \F_2^n}[F_{y_1,\ldots,y_d}(x)] \right)^{1/2^d}.
$$
\end{dfn}

The following summarize some simple facts regarding  the Gowers norm.

\begin{fact}[Simple facts regarding the Gowers norm]\label{lem:gowers}
Let $f:\F_2^n \to \F_2$.
\begin{enumerate}
\item $\|\cdot\|_{U^d}$ is a norm of complex functions (for $d=1$
it is a seminorm).
\item $0 \le \|(-1)^f\|_{U^d} \le 1$.
\item $\|(-1)^f\|_{U^d}=1$ iff $f$ is a polynomial of degree at most $d-1$.
\item If $f$ is a random boolean function then $\|(-1)^f\|_{U^d} \approx 0$.
\item Assume there is a polynomial $p(x)$ of degree at most $d-1$ such that $\Pr_x[f(x)=p(x)] \ge \frac{1+\epsilon}{2}$. Then $\|(-1)^f\|_{U^d} \ge \epsilon$.
\end{enumerate}
\end{fact}

The hard direction is proving structure theorems for functions
with noticeable Gowers norm. This is known as the {\em inverse
Gowers conjecture}. Let $f:\F_2^n \to \F_2$ be a function for
which $\|(-1)^f\|_{U^d} \ge \epsilon$. The conjecture speculates
there exists a polynomial $p(x)$ of degree at most $d-1$ for which
$\Pr_x[f(x)=p(x)] \ge \frac{1}{2}+\epsilon'$, where $\epsilon'$
may depend on $\epsilon$ and $d$, but crucially it does not depend
on the number of variables $n$. Much is known today about the
inverse Gowers conjecture. We summarize below the current state of
affairs.

\begin{fact}[Inverse Gowers conjecture]\label{thm:igc}
Let $f:\F_2^n \to \F_2$ such that $\|f\|_{U^d} \ge \epsilon$ for
$d \le 3$. Then there exists a polynomial $p(x)$ of degree at most
$d-1$ such that $\Pr[f(x)=p(x)] \ge \tfrac{1}{2}+\epsilon'$, where
\begin{itemize}
\item $d=1$: It is easy to verify that $\|(-1)^f\|_{U^1} = |\E[(-1)^f]|$. This gives $\epsilon' = \epsilon/2$.
\item $d=2$: It can be shown by simple Fourier analysis that the $U^2$ norm of $(-1)^f$ is equal to the $L_4$ norm of the Fourier coefficients of $f$, that is $\|f\|_{U^2} = \|\hat{f}\|_4$. This  gives $\epsilon' \ge \Omega(\epsilon^2)$.
\item $d=3$: This case is more involved. Results of Green and Tao~\cite{GreenTao08:inverse_U3} and
Samorodnitsky~\cite{Samorodnitsky07} give that in this case $\epsilon' \ge exp(-1/\epsilon)$.
\end{itemize}
When $d \ge 4$ things become trickier. It is no longer true that if $\|f\|_{U^d} \ge \epsilon$
there must exist a polynomial $p:\F_2^n \to \F_2$ of degree at most $d-1$ which approximates $f$
with probability noticeably larger than $1/2$~\cite{LovettMeSa08,GreenTao07:polynomials_finite_fields}.
Nevertheless, a refined inverse conjecture holds: there exists a "non-classical" polynomial $F:\F_2^n \to \C$ which
approximates $f$. The notion of approximation is the natural generalization of our
previous definition,  $\E_x[(-1)^{f(x)} F(x)] \ge \epsilon'$. A "non-classical" polynomial is a function
$F:\F_2^n \to \C$ for which $F_{y_1,\ldots,y_d}(x) \equiv 1$ (for example, $F(x)=i^{x_1+\ldots+x_n}$ is a "non-classical" quadratic). This was proved by Bergelson, Tao and Ziegler~\cite{BergelsonTaZi09, TaoZi09} using Ergodic theory. A major caveat of this approach is that currently
no explicit bound on $\epsilon'$ in terms of $\epsilon$ and $d$ is known. All that is known is that $\epsilon'$
is some constant depending only on $\epsilon,d$.
\end{fact}

The only case where there is an explicit relation between $\epsilon'$ and $\epsilon$ which is not polynomial is the case of $\|\cdot\|_{U^3}$, where it is believed to be suboptimal. The following polynomial relation
is conjectured.

\begin{conj}[Polynomial Inverse Gowers conjecture for $U^3$]\label{conj:IGC3}
Let $f:\F_2^n \to \F_2$. If $\|f\|_{U^3} \ge \epsilon$ then
there exists a quadratic polynomial $p(x)$ such that $\Pr[f(x)=p(x)] \ge \tfrac{1}{2} + \epsilon^{O(1)}$.
\end{conj}

Our main result is that the Polynomial Freiman-Ruzsa conjecture
and the Polynomial Inverse Gowers conjecture for $U^3$ are
equivalent.

\begin{thm}\label{thm:main}
Conjecture~\ref{conj:pfr} and Conjecture~\ref{conj:IGC3} are equivalent.
\end{thm}

One direction is simple. The only place in the proof of the
inverse conjecture for $U^3$ where a super-polynomial loss occurs
is in the use of the Ruzsa theorem. Assuming the polynomial
Freiman-Ruzsa conjecture this loss can be avoided. We sketch the
required change in the proof in Section~\ref{sec:pfr_to_u3}.

The main innovation of this work is a proof of the polynomial
Freiman-Ruzsa conjecture assuming a polynomial inverse theorem for
$U^3$. We prove this in Section~\ref{sec:u3_to_pfr}.

We note that this result was also independently discovered by Green and Tao~\cite{greentao09:equiv}.

\section{Deducing the polynomial Freiman-Ruzsa conjecture, assuming a polynomial
inverse conjecture for $U^3$}\label{sec:u3_to_pfr}

We will prove Conjecture~\ref{conj:structured_approx_hom}, which
is equivalent to the polynomial Freiman-Ruzsa conjecture. Let
$f:\F_2^n \to \F_n^m$ be a function and let $\Delta f = \{f(x+y)-f(x)-f(y):x,y
\in \F_2^n\}$. We assume $|\Delta f| \le K$, and wish to prove that
there exists a linear map $\ell:\F_2^n \to \F_2^m$ such that
$|\{f(x)-\ell(x): x \in \F_2^n\}| \le K^{O(1)}$.

Define a function $F:\F_2^{n+m} \to \F_2$ by $F(x,z) =
\ip{f(x),z}$ for $x \in \F_2^n, z \in \F_2^m$, where $\ip{\cdot,\cdot}$ denotes inner product. The proof will
proceed in the following steps.
\begin{enumerate}
\item Show that $\|F\|_{U^3} \ge K^{-O(1)}$.
\item By the polynomial inverse conjecture for $U^3$, there is a
quadratic polynomial $Q(x,z)$ such that $\Pr[F(x,z)=Q(x,z)] \ge
\tfrac{1}{2} + K^{-O(1)}$.
\item Deduce there is a linear map $\ell:\F_2^n \to \F_2^m$ such that $\Pr[f(x)=\ell(x)+c] \ge K^{-O(1)}$ for some $c \in \F_2^m$.
\item Conclude by showing that $|\{f(x)-\ell(x): x \in \F_2^n\}| \le K^{O(1)}$.
\end{enumerate}

\begin{lemma}\label{lem:F_U3}
$\|F\|_{U^3} \ge K^{-7/8}$.
\end{lemma}

\begin{proof}
We compute $\|F\|_{U^3}^8$. Let $x,y_1,y_2,y_3 \in \F_2^n$ and $z,w_1,w_2,w_3 \in \F_2^m$ be chosen uniformly. We have
\begin{align*}
\|F\|_{U^3}^8 &= \E[(-1)^{\sum_{I \subseteq [3]} F(x + \sum_{i \in I}y_i, z + \sum_{i \in I}w_i)}] \\
&= \E[(-1)^{\sum_{I \subseteq [3]} \ip{f(x + \sum_{i \in I}y_i),z + \sum_{i \in I}w_i}}] \\
& = \E[(-1)^{\ip{z,A_0}+\ip{w_1,A_1}+\ip{w_2,A_2}+\ip{w_3,A_3}}]
\end{align*}
where
\begin{align*}
&A_0 = \sum_{I \subseteq \{1,2,3\}} f(x + \sum_{i \in I} y_i)\\
&A_1 = \sum_{I \subseteq \{2,3\}} f(x+y_1+\sum_{i \in I} y_i)\\
&A_2 = \sum_{I \subseteq \{1,3\}} f(x+y_2+\sum_{i \in I} y_i)\\
&A_3 = \sum_{I \subseteq \{1,2\}} f(x+y_3+\sum_{i \in I} y_i)\\
\end{align*}
Hence we have
\begin{align*}
\|F\|_{U^3}^8 & = \E[(-1)^{\ip{z,A_0}+\ip{w_1,A_1}+\ip{w_2,A_2}+\ip{w_3,A_3}}] \\
&= Pr_{x,y_1,y_2,y_3 \in \F_2^n}[A_0=0, A_1=0, A_2=0, A_3=0].
\end{align*}

The proof will follow from the following general claim.

\begin{claim}\label{clm:large_set}
For any  $k \ge 1$ there exist values $c_1,\ldots,c_k \in \F_2^m$
such that the set
$$
S_k=\{(x,y_1,\ldots,y_k) \in (\F_2^n)^{k+1}: \forall I \subseteq
[k],\ f(x + \sum_{i \in I}y_i) = f(x) + \sum_{i \in I} f(y_i) +
\sum_{i \in I} c_i\}
$$
has relative size at least $\frac{|S_k|}{2^{n(k+1)}} \ge
(1/K)^{2^k-1}$.
\end{claim}

Before proving the claim we show how it can be applied to conclude the proof of Lemma~\ref{lem:F_U3}.
Let $c_1,c_2,c_3 \in \F_2^m$ be values and let
$$
S_{3}=\{(x,y_1,y_2,y_3) \in (\F_2^n)^4: \forall I \subseteq [3],\ f(x + \sum_{i \in I}y_i) = f(x) + \sum_{i \in I} f(y_i) + \sum_{i \in I} c_i\}.
$$
such that its relative size is $\frac{|S_{3}|}{2^{4n}} \ge (1/K)^7$. Notice that if $(x,y_1,y_2,y_3) \in S_{3}$ then
$A_0=A_1=A_2=A_3=0$ as each variable appears an even number of times in each of $A_0,A_1,A_2,A_3$.
Thus we conclude that
$$
\|F\|_{U^3}^8 = \Pr[A_0=A_1=A_2=A_3=0] \ge (1/K)^7.
$$
\end{proof}

We now turn  to prove the claim.
\begin{proof}[Proof of Claim~\ref{clm:large_set}]
The proof will be by induction on $k$. For $k=1$ this follows since $f(x+y)-f(x)-f(y) \in \Delta$ for all $x,y \in \F_2^n$ and $|\Delta| \le K$. Assume the claim holds for $k$, and we will prove it for $k+1$. Let $c_1,\ldots,c_k \in \F_2^m$ be such that
$$
S_k=\{(x,y_1,\ldots,y_k) \in (\F_2^n)^{k+1}: \forall I \subseteq
[k],\ f(x + \sum_{i \in I}y_i) = f(x) + \sum_{i \in I} f(y_i) +
\sum_{i \in I} c_i\}
$$
has relative size at least $\frac{|S_k|}{2^{n(k+1)}} \ge
(1/K)^{2^k-1}$. Consider the set
$$
S'=\{(x,y_1,\ldots,y_k,y_{k+1}) \in (\F_2^n)^{k+2}:
(x,y_1,\ldots,y_k) \in S_k \textrm{ and }
(x+y_{k+1},y_1,\ldots,y_k) \in S_k\}
$$
By the Cauchy-Schwartz inequality its relative size is lower bounded by $(1/K)^{2^k-2}$, as
\begin{align*}
\frac{|S'|}{2^{n(k+2)}} &= \E_{x,y_1,\ldots,y_k,y_{k+1}}[\mathbf{1}_{(x,y_1,\ldots,y_k) \in S_k} \cdot \mathbf{1}_{(x+y_{k+1},y_1,\ldots,y_k) \in S_k}] \\
&= \E_{x,z,y_1,\ldots,y_k}[\mathbf{1}_{(x,y_1,\ldots,y_k) \in S_k} \cdot \mathbf{1}_{(z,y_1,\ldots,y_k) \in S_k}] \\
& = \E_{y_1,\ldots,y_k}[\E_x[\mathbf{1}_{(x,y_1,\ldots,y_k) \in S_k}]^2] \\
& \ge (\E_{x,y_1,\ldots,y_k}[\mathbf{1}_{(x,y_1,\ldots,y_k) \in S_k}])^2 \\
&= \left( \frac{|S_k|}{2^{n(k+1)}} \right)^2 \ge (1/K)^{2^k-2}.
\end{align*}
Fix $c_{k+1} \in \Delta$ such that $\Pr_{(x,y_1,\ldots,y_{k+1} )\in S'}[f(x+y_{k+1})-f(x)-f(y_{k+1})=c_{k+1}] \ge 1/K$.
The required set $S_{k+1}$ is chosen to be\begin{align*}
S_{k+1} = \{(x,y_1,\ldots,y_{k+1}) \in (\F_2^n)^{k+2}: &(x,y_1,\ldots,y_{k+1}) \in S' \textrm{ and }\\ &f(x+y_{k+1})-f(x)-f(y_{k+1})=c_{k+1}\}.
\end{align*}
Observe that the relative size of $S_{k+1}$ is as required,
$$
\frac{|S_{k+1}|}{2^{n(k+2)}} \ge \frac{1}{K} \cdot
\frac{|S'|}{2^{n(k+2)}} \ge (1/K)^{2^{k+1}-1}.
$$
Let $(x,y_1,\ldots,y_{k+1}) \in S_{k+1}$. We need to show that
$f(x+\sum_{i \in I}y_i) = f(x) + \sum_{i \in I}y_i + \sum_{i \in
I} c_i$. If $(k+1) \notin I$ this follow by induction from the
assumption on $S_k$. Otherwise, using the fact that
$(x+y_{k+1},y_1,\ldots,y_k) \in S_k$ we get that
$f(x+y_{k+1}+\sum_{i \in I}y_i) = f(x+y_{k+1}) + \sum_{i \in I \setminus \{k+1\}}f(y_i)
+ \sum_{i \in I \setminus \{k+1\}} c_i$, and since
$f(x+y_{k+1})=f(x)+f(y_{k+1})+c_{k+1}$ we conclude the proof.
\end{proof}

Using Lemma~\ref{lem:F_U3} and the polynomial inverse conjecture for the Gowers $U^3$ norm, we get
there is a quadratic polynomial such that $\Pr[\ip{f(x),z}=Q(x,z)] \ge \tfrac{1}{2}+K^{-O(1)}$.

\begin{lemma}\label{lem:lin_approx}
Let $Q(x,z)$ be a quadratic polynomial such that $\Pr[\ip{f(x),z}=Q(x,z)] \ge \tfrac{1}{2}+\epsilon$.
Then there exist a linear mapping $\ell:\F_2^n \to \F_2^m$ and a constant $c \in \F_2^m  $ such that
$$
\Pr[f(x)=\ell(x)+c] \ge \epsilon^4/K.
$$
\end{lemma}

\begin{proof}
Let $Q(x,z) = x^T A z + Q_1(x) + Q_2(z)$, where $A$ is an $n \times m$ matrix and $Q_1,Q_2$ are quadratics
just in the $x$ and $z$ variables. We have
$$
|\E[(-1)^{\ip{f(x),z}+Q(x,z)}]| = 2 \Pr[\ip{f(x),z}=Q(x,z)]-1 \ge \epsilon.
$$
We will use the following simple claim, which follows by applying the Cauchy-Schwartz inequality twice: for any  function $G(x,z)$ the following holds$$
G(x,z)^4 \le \E_{x',x'',z',z''} G(x',z')G(x'',z')G(x',z'')G(x'',z'').
$$
Let $G(x,z)=(-1)^{\ip{f(x),z}+Q(x,z)}$. Note that
$$
G(x',z')G(x'',z')G(x',z'')G(x'',z'') = (-1)^{\ip{f(x')+f(x''),z'+z''} + (x'+x'')^T A (z'+z'')}
$$
hence we deduce, by setting $w=z'+z''$ that
$$
\E_{x',x'',w}[(-1)^{\ip{f(x')+f(x''),w}+(x'+x'')^T A w}] \ge \epsilon^4.
$$
Let $\ell:\F_2^n \to \F_2^m$ be the linear mapping defined by $A$, that is  $\ell(x) = x^T A$. Note that we have
$$
\E_{x',x'',w}[(-1)^{\ip{f(x')+f(x'')+\ell(x'+x''),w}}] \ge \epsilon^4,
$$
hence
$$
\Pr_{x',x''}[f(x')+f(x'')=\ell(x'+x'')] \ge \epsilon^4.
$$
We are nearly done. To complete the proof we use the fact that $f(x'+x'')-f(x')-f(x'') \in \Delta f$
to deduce that there is some $c \in \Delta f$ such that
$$
\Pr_{x',x''}[f(x'+x'')=f(x')+f(x'')+c|f(x')+f(x'')=\ell(x'+x'')] \ge 1/K,
$$
hence we got the required result,
$$
\Pr_{x}[f(x)=\ell(x)+c] \ge \epsilon^4/K.
$$
\end{proof}

We finish the proof by a standard covering argument.

\begin{lemma} Assume there is a linear map $\ell:\F_2^n \to \F_2^m$ and $c \in \F_2^n$ such that
$\Pr[f(x)=\ell(x)+c] \ge \epsilon$. Then $|\{f(x)-\ell(x): x \in \F_2^n\}| \le K^2/\epsilon$.
\end{lemma}

\begin{proof}
Let $T=\{x \in \F_2^n: f(x)=\ell(x)+c\}$. Let $B \subset \F_2^n$ be maximal such that
for any distinct $b',b'' \in B$ the sets $T+b'$ and $T+b''$ are disjoint. Clearly $|B| \le 1/\epsilon$.
Let $x \in \F_2^n$ be arbitrary. By the maximality of $B$ we have that $(T+x) \cap (T+B) \ne \emptyset$, hence we get that $x \in T+T+B$. Let $x=t'+t''+b$ for $t',t'' \in T$ and $b \in B$. We have $f(x)=f(t')+f(t'')+f(b)+r$ for $r \in \Delta f+\Delta f$. Thus we have
\begin{align*}
f(x)-\ell(x)&=f(t'+t''+b)-\ell(t'+t''+b) \\
& = (f(t')-\ell(t'))+(f(t'')-\ell(t''))+(f(b)-b)) + r \\
&= c+c+(f(b)-b)+r.
\end{align*}
Let $B' = \{f(b)-b: b \in B\}$. We got that $\{f(x)-\ell(x): x \in \F_2^n\} \subset \Delta f+\Delta f+B'$, and
an obvious upper bound is $|\Delta f+\Delta f+B'| \le |\Delta f|^2 |B'| \le K^2/\epsilon$.
\end{proof}

\section{Deducing a polynomial inverse theorem for $U^3$,
assuming the polynomial Freiman-Ruzsa
conjecture}\label{sec:pfr_to_u3}

We follow the proof of the inverse theorem for $U^3$ of
Samorodnitsky~\cite{Samorodnitsky07}. Let $f:\F_2^n \to \F_2$ be a
function such $\|(-1)^f\|_{U^3} \ge \epsilon$. The proof proceeds
as follows.

\begin{enumerate}
\item For $\epsilon^{O(1)}$ fraction of $y \in \F_2^n$ we have that $\|(-1)^{f_y}\|_{U^2} \ge
\epsilon^{O(1)}$ (Corollary 6.2).
\item Using the inverse theorem for $\|\cdot\|_{U^2}$, there exist
linear maps $\ell^{(y)}:\F_2^n \to \F_2$ such that
$\Pr_{x,y}[f_y(x) = \ell^{(y)}(x)] \ge \epsilon^{O(1)}$.
\item Crucially, one can choose the linear maps to
behave linearly in $y$, $\Pr_{y,z}[\ell^{(y+z)} \equiv \ell^{(y)}
+ \ell^{(z)}] \ge \epsilon^{O(1)}$ (Lemma 6.7).
\item Define $S=\{(y,\ell^{(y)}): y \in \F_2^n\} \subset \F_2^{2n}$.
We have $\Pr_{a,b \in S}[a+b \in S] \ge \epsilon^{O(1)}$.
\item By the Balog-Szemer\`{e}di-Gowers theorem there exists $S'
\subset S$ such that $|S'| \ge \epsilon^{O(1)}|S|$ and $|S'+S'|
\le \epsilon^{-O(1)}$.
\item Originally, Ruzsa's theorem was used to deduce that $Span(S')
\le exp(1/\epsilon^4) |S'|$. We replace it by the polynomial
Freiman-Ruzsa conjecture to deduce there is $S'' \subset S'$ such
that $|S''| \ge \epsilon^{O(1)} |S'|$ and $|Span(S'')| \le
\epsilon^{-O(1)} |S''|$.
\item $S''$ can be used to construct a global linear map $L:\F_2^n \to
\F_2^n$ such that $\Pr_{y}[\ell^{(y)} = L(y)] \ge \epsilon^{O(1)}$
(Discussion following Theorem 6.9).
\item The remainder of the proof shows how to integrate $L$ to find a
quadratic $q(x)$ such that $\Pr[q(x)=f(x)] \ge \tfrac{1}{2} +
\epsilon^{O(1)}$ (Lemmas 6.10 and 6.11).
\end{enumerate}

{\em Acknowledgement.}  I would like to thank Amir Shpilka and Partha
Mukhopadhyay for invaluable discussions. I would also like to thank
Alex Samorodnitsky and Seva Lev for early discussions on this problem.

\end{document}